\documentclass[12pt]{amsart}
\usepackage{amssymb,latexsym}
\usepackage{amsfonts}
\usepackage{amsmath}
\usepackage[colorlinks,linkcolor=blue,anchorcolor=blue,citecolor=blue]{hyperref}
\usepackage{algorithm}
\usepackage{enumerate}
\usepackage{algpseudocode}
\usepackage{verbatim}
\usepackage{graphicx}

\newcommand\F{{\mathbb F}}

\newcommand\Z{{\mathbb Z}}

\newcommand\cO{{\mathcal O}}

\newtheorem{theorem}{Theorem}[section]
\newtheorem{lemma}[theorem]{Lemma}
\newtheorem{corollary}[theorem]{Corollary}

\theoremstyle{definition}
\newtheorem{definition}[theorem]{Definition}

\theoremstyle{remark}

\numberwithin{equation}{section}

\begin{document}
\baselineskip=17pt

\title[Carmichael rings and Carmichael polynomials]{On the Carmichael rings, Carmichael ideals and Carmichael polynomials}
\author{Sunghan Bae}
\address{Department of Mathematical Sciences, Korea Advanced Institute of Science and Technology (KAIST), Daejeon 305-701, Republic of Korea}
\email{shbae@kaist.ac.kr}

\author{Su Hu}
\address{Department of Mathematics, South China University of Technology, Guangzhou 510640, China}
\email{mahusu@scut.edu.cn}

\author{Min Sha}
\address{Department of Computing, Macquarie University,
 Sydney, NSW 2109, Australia}
 \email{shamin2010@gmail.com} 
 

\maketitle

\begin{abstract} 
Motivated by Carmichael numbers, 
we say that  a finite ring $R$ is a Carmichael ring if $a^{|R|}=a$ for any $a \in R$. 
We then call an ideal $I$ of a ring $R$ as a Carmichael ideal if $R/I$ is a Carmichael ring, 
and a Carmichael element of $R$ means it generates a Carmichael ideal. 
In this paper, we determine the structure of Carmichael rings 
and prove a generalization of Korselt's criterion for Carmichael ideals in Dedekind domains. 
We extend several results from the number field case to the function field case. 
Especially, we study Carmichael elements of polynomial rings over finite fields (called Carmichael polynomials) 
by generalizing some classical results. 
For example, we show that there are infinitely many Carmichael polynomials but they have zero density. 
\end{abstract}

\section{Introduction}

\subsection{Background and motivation}

By Fermat's Little Theorem, we know that if $p$ is a prime number,  
then $a^{p} \equiv a~ ({\rm mod}~ p)$ for any integer $a\in\mathbb{Z}$. 
However, the converse is not true, we call such exceptional integers as Carmichael numbers. 

\begin{definition}  \label{def:Car-num}
A composite integer $n$ is
called a \textit{Carmichael number} if $a^{n}\equiv a~({\rm mod}~ n)$ for any integer $a\in\mathbb{Z}$. 
\end{definition}

The first ten Carmichael numbers are (see the sequence A002997 in the OEIS \cite{Sloane}): 
$$
561, 1105, 1729, 2465, 2821, 6601, 8911, 10585, 15841, 29341. 
$$
One can completely characterize
all Carmichael numbers using Korselt's criterion.

\begin{theorem}[Korselt's criterion]   \label{thm:Korselt}
A composite integer $n$ is Carmichael if and only if $n$ is
square-free and $p-1 \mid n-1$ for any prime $p \mid n$.
\end{theorem}

In 1953, Kn{\"o}del \cite{Knodel} gave an upper bound for the number of Carmichael numbers up to $x$, 
which was improved by Erd{\"o}s \cite{Erdos} later on. 
In 1994, Alford, Granville and  Pomerance \cite{AGP} proved that there exist
infinitely many Carmichael numbers by providing a lower bound; see \cite{Harman05,Harman08} for some further improvements. 
Moreover, Wright \cite{Wright13} showed that there are infinitely many Carmichael numbers in each arithmetic progression 
$a$ modulo $d$ for positive integers $a,d$ with $\gcd(a,d)=1$; see \cite{BP,Mat} for some previous results. 
Recently, Wright \cite{Wright16} proved that  for some fixed integer $m$, 
there are infinitely many Carmichael numbers with exactly $m$ prime factors; in fact, there are infinitely many such $m$. 

Recently, Steele \cite{Steele} generalized Carmichael numbers to
  ideals in number fields and proved a generalization of Korselt's
criterion for these Carmichael ideals.  
 He also showed that for any composite integer $n$, 
 there are infinitely many abelian number fields $K$
with discriminant relatively prime to $n$ such that $n$ does not generate a 
Carmichael ideal in $K$. 
Besides, Schettler \cite{Sch} generalized Carmichael numbers to elements in a principal ideal domain. 

In this paper, we want to generalize Carmichael numbers in a more general setting 
including the generalizations of Steele and Schettler as special cases, 
and then extend various classical or recent results about Carmichael numbers. 

\subsection{Our considerations}

Following the definition of  Carmichael number (Definition \ref{def:Car-num}), we first introduce Carmichael ring. 

\begin{definition}  \label{def:Car-ring}
A finite ring $R$ is called a \textit{Carmichael ring} if it is not a field 
and $a^{|R|}=a$ for any $a \in R$. 
\end{definition}

If $n$ is a Carmichael number, then by definition  the residue class ring $\Z/n\Z$ is a Carmichael ring. 
By  a classical result of Jacobson (see \cite[Theorem 11]{Jacobson}),
  Carmichael rings are automatically commutative rings. 

We shall determine the structure of Carmichael rings in Theorem~\ref{thm:Car-structure}, 
which can be viewed as a generalization of Korselt's criterion.  

It also seems naturally to introduce Carmichael ideal and Carmichael element of a ring. 

\begin{definition}  \label{def:Car-ideal}
An ideal $I$ of a ring $R$ is  said to be a \textit{Carmichael ideal} 
if  $R/I$ is a Carmichael ring. An element of $R$ is called a \textit{Carmichael element} if it generates a Carmichael ideal. 
\end{definition}

We prove a generalization of Korselt's criterion for Carmichael ideals in Dedekind domains in Theorem~\ref{thm:Ded-Korselt} 
and also study the  behavior of Carmichael ideals in the extensions of Dedekind domains. 

We then consider Carmichael elements in polynomial rings over finite fields and in function fields 
in Sections~\ref{sec:Car-pol} and \ref{sec:Car-ele} respectively. 

Throughout the paper, let $\F_q$ be the finite field of $q$ elements, and $\F_q[t]$ the polynomial ring of one variable over $\F_q$. 
Following Definition~\ref{def:Car-ideal}, a polynomial $g$ in $\F_q[t]$ is called a \textit{Carmichael polynomial} 
if $g$ generates a Carmichael ideal in $\F_q[t]$.  

We remark here that Hsu \cite{Hsu} introduced another concept of Carmichael polynomials by using Carlitz modules, 
which is also a generalization of Carmichael numbers. 
The difference is that when analogizing ``$a^{n}\equiv a~({\rm mod}~ n)$" for $\mathbb{F}_{q}[t]$, 
Hsu views the $n$ in $a^n$ as an element of the integer ring $\Z$ and $a^{n}$ as ``$n$ acts on $a$'', 
but we view it as the cardinality of the residue class ring $\Z/n\Z$.  

In this paper, we extend various results about Carmichael numbers to Carmichael polynomials. 
For example, we establish the Korselt criterion for these polynomials (see Theorem~\ref{thm:pol-Korselt}), 
and we obtain lower and upper bounds for the number of monic Carmichael polynomials of fixed degree 
(see Theorems~\ref{thm:lower} and \ref{thm:upper}). 
Then, one can see that they have zero density.  

Especially, we find two properties which do not hold for Carmichael numbers. 
The first one is that any square-free polynomial in $\F_q[t]$ is a factor of infinitely many Carmichael polynomials (see Theorem~\ref{thm:const}). 
The other is that any Carmichael polynomial $g$ remains Carmichael in any finite Galois extension over $\F_q(t)$ 
with discriminant relatively prime to $g$ (see Theorem~\ref{thm:ext-Car1}).

\section{Carmichael  ring} 

In this section, we determine the structure of Carmichael ring, 
which  implies  the classical Korselt's criterion  (Theorems \ref{thm:Korselt}) 
and its generalization in Theorem \ref{thm:Ded-Korselt}.

\begin{theorem}[The structure theorem of Carmichael ring]\label{thm:Car-structure} 
Let $R$ be a finite ring with identity.  
Then,  $R$ is a Carmichael ring if and only if 
$$
R  \cong \F_{q_1} \times \cdots \times \F_{q_k}
$$ 
for some integer $k \ge 2$, and for each $1 \le i \le k$, $\F_{q_i}$ is a finite field of $q_i$ elements and  $q_{i}-1 \mid |R|-1$.  
\end{theorem}

\begin{proof} 
We only need to prove the necessity. Assume that $R$ is a Carmichael ring. 
Then, automatically $R$ is a commutative ring. 

Consider the natural homomorphism 
$$
\sigma:  R \to \prod_{\mathfrak{M}} R/\mathfrak{M}, \quad a \mapsto (a, \ldots, a), 
$$
where $\mathfrak{M}$ runs through all the maximal ideals of $R$. 

For $a \in R$, if $\sigma(a)=0$, then $a \in \mathfrak{M}$ for each maximal ideal $\mathfrak{M}$ of $R$. 
Besides, since $a^{|R|} = a$ by definition, we have $(1-a^{|R|-1})a=0$. 
If $1-a^{|R| -1}$ is not a unit, then there eixsts a maximal ideal, say $\mathfrak{M}_0$, such that $1-a^{|R| -1} \in \mathfrak{M}_0$, 
and so $1 \in \mathfrak{M}_0$ (because $a \in \mathfrak{M}_0$), which leads to a contradiction. 
So, we must have that $1-a^{|R|-1}$ is a unit, and thus $a=0$. 
Hence, $\sigma$ is injective. 
Noticing that $R$ has only finitely many maximal ideals and using the Chinese remainder theorem, 
we know that $\sigma$ is also surjective, and thus it is an isomorphism. 
Since $R$ is not a field, it must have more than one maximal ideals. 

Moreover, each $R/\mathfrak{M}$ is in fact a finite field. 
Due to $a^{|R|}=a$ for any $a \in R/\mathfrak{M}$, 
we see that $|R/\mathfrak{M}|-1$ divides $|R|-1$. 
This completes the proof. 
\end{proof}

The following corollary suggests that there exist finite rings $R$ 
such that any non-trivial ideal of $R$ is not a Carmichael ideal. 

\begin{corollary}
Let $\F_{q_1},\F_{q_2}, \F_{q_3}$ be three distinct finite fields, 
and let $R = \F_{q_1} \times \F_{q_2} \times \F_{q_3}$.  
Then, any non-trivial ideal of $R$ is not a Carmichael ideal. 
\end{corollary}

\begin{proof}
Note that a field has only trivial ideals, and a Carmichael ring is not a field. 
We only need to consider the ideals of $R$ isomorphic to $\F_{q_1},\F_{q_2}, \F_{q_3}$. 
So, it suffices to show that the following rings are not Carmichael rings: 
$$
\F_{q_1} \times \F_{q_2}, \quad \F_{q_1} \times \F_{q_3}, \quad \F_{q_2} \times \F_{q_3}. 
$$
For example, consider the ring $\F_{q_1} \times \F_{q_2}$, if it is a Carmichael ring, 
then by Theorem~\ref{thm:Car-structure} we have 
$$
q_1 -1 \mid q_1q_2 - 1, \qquad  q_2 -1 \mid q_1q_2 - 1, 
$$
which implies $q_1=q_2$. 
This contradicts with the assumption that $\F_{q_1}$ and $\F_{q_2}$ are two distinct finite fields. 
\end{proof}

\section{Carmichael ideals in Dedekind domains}

In this section, we consider Carmichael ideals in Dedekind domains by generalizing some results in \cite{Steele}. 

Suppose that $\mathcal{O}_{K}$ is a Dedekind domains, and $K$ is the
fraction field of $\mathcal{O}_{K}$. For any ideal $\mathfrak{n}$ of
$\mathcal{O}_{K}$, denote
$$
N_{K}(\mathfrak{n})=|\mathcal{O}_{K}/\mathfrak{n}|.
$$ 
From Definition~\ref{def:Car-ideal}, an ideal $\mathfrak{n}$ of $\cO_K$  is a Carmichael ideal 
if and only if $\mathfrak{n}$ is a composite ideal, $N_{K}(\mathfrak{n})$ is finite, and for all
$\alpha$ in $\mathcal{O}_{K}$, we have
$\alpha^{N_{K}(\mathfrak{n})}\equiv\alpha ~(\textrm{mod}~ \mathfrak{n})$. 

Using Theorem~\ref{thm:Car-structure}, it is easy to get a necessary
and sufficient condition for an ideal to be a Carmichael ideal in
$\mathcal{O}_{K}$, generalizing Theorem~\ref{thm:Korselt} and also 
 Korselt's criterion in number field case (see \cite[Theorem 2.2]{Steele}).

\begin{theorem}[Korselt's criterion for Dedekind domains]  \label{thm:Ded-Korselt} 
A composite ideal $\mathfrak{n}$ is a Carmichael ideal of $\mathcal{O}_{K}$ if and only if 
\begin{itemize}
\item[(1)] $\mathfrak{n}$ is square-free,
\item[(2)] $N_{K}(\mathfrak{n})$ is finite,
\item[(3)] $N_{K}(\mathfrak{p})-1$ divides $N_{K}(\mathfrak{n})-1$ for any prime ideal $\mathfrak{p} \mid \mathfrak{n}$.
\end{itemize}
\end{theorem}

\begin{proof}
Suppose that $\mathfrak{n}$ has the prime factorization: 
$$
\mathfrak{n}=\mathfrak{p}_{1}^{e_{1}}\mathfrak{p}_{2}^{e_{2}}\cdots\mathfrak{p}_{s}^{e_{s}},
$$
where each $\mathfrak{p}_{i}$, $1\leq i\leq s$, is a prime ideal of $\mathcal{O}_{K}$. 
By the Chinese Reminder Theorem, we have
$$
\mathcal{O}_{K}/\mathfrak{n}=\mathcal{O}_{K}/\mathfrak{p}_{1}^{e_{1}} \times 
\mathcal{O}_{K}/\mathfrak{p}_{2}^{e_{2}} \times 
\cdots \times \mathcal{O}_{K}/\mathfrak{p}_{s}^{e_{s}}.
$$ From
Theorem~\ref{thm:Car-structure}, we get what we want.
\end{proof}

We now consider Carmichael ideals in the extensions of Dedekind domains. 
By Theorem~\ref{thm:Ded-Korselt} we only need to consider square-free ideals. 

\begin{theorem}~\label{thm:Ded-ideal}
Suppose that $L$ is a finite separable extension over $K$ of degree $d$,
$\mathfrak{n}$ is a square-free ideal of $\cO_K$, and $N_{K}(\mathfrak{n})$ is finite. 
Let $\mathcal{O}_{L}$ be the integral closure of $\mathcal{O}_{K}$ in $L$. 
Then, $\mathfrak{n}\cO_L$ is Carmichael in $\mathcal{O}_{L}$ if and only if 
\begin{itemize}
\item[(1)] $\mathfrak{n}\cO_L$ is a composite ideal,
\item[(2)] $\mathfrak{n}$ is relatively prime to the discriminant ${\rm Disc}(L/K)$,
\item[(3)] for each prime ideal $\mathfrak{p}\mid \mathfrak{n}$ and any prime ideal $\mathfrak{P}$ of $\cO_L$  
lying above $\mathfrak{p}$, we have
$N_{K}(\mathfrak{p})^{f(\mathfrak{P})}-1\mid
N_{K}(\mathfrak{n})^{d}-1$, where $f(\mathfrak{P})$ is the residue
class degree of $\mathfrak{P}$ in $L/K$.
\end{itemize}
\end{theorem}

\begin{proof}
We first prove the necessity by using some basic properties of Dedekind domains.  
Since $\mathfrak{n}\cO_L$ is Carmichael in $\mathcal{O}_{L}$, 
by Theorem~\ref{thm:Ded-Korselt} we have that $\mathfrak{n}\cO_L$ is a composite and square-free ideal. 
That is,  all the prime factors of $\mathfrak{n}$ are  unramified in $L/K$, 
which means that $\mathfrak{n}$ is relatively prime to the discriminant ${\rm Disc}(L/K)$. 
Besides, for each prime ideal $\mathfrak{p}\mid \mathfrak{n}$ and any prime ideal $\mathfrak{P}$ of $\cO_L$  
lying above $\mathfrak{p}$, by Theorem~\ref{thm:Ded-Korselt} we have that 
$N_{L}(\mathfrak{P})-1$ divides $N_{L}(\mathfrak{n}\cO_{L})-1$. 
We complete the proof of this part by noticing 
$N_{L}(\mathfrak{P})=N_{K}(\mathfrak{p})^{f(\mathfrak{P})}$ and 
 $N_{L}(\mathfrak{n}\cO_{L})= N_{K}(\mathfrak{n})^{d}$. 

Conversely, one can prove the sufficiency directly by using Theorem~\ref{thm:Ded-Korselt}.
\end{proof}

As in \cite[Theorem 2.3]{Steele}, the following is a generalization of
Fermat's Little Theorem to the case of Dedekind domains.

\begin{corollary}  \label{cor:Ded-prime}
Let $L$ be a finite Galois extension of $K$ . 
Suppose that $\mathfrak{p}$ is a non-zero prime ideal of $\mathcal{O}_{K}$, $N_{K}(\mathfrak{p})$ is finite,  
and $\mathfrak{p}$ does not divide the discriminant ${\rm Disc}(L/K)$. 
Then, we have 
$$
\alpha^{N_{L}(\mathfrak{p}\cO_{L})}\equiv\alpha ~({\rm mod}~\mathfrak{p}\cO_L)
$$
for all $\alpha\in \mathcal{O}_{L}$. 
That is,  the ideal $\mathfrak{p}\mathcal{O}_{L}$ is either prime or
Carmichael.
\end{corollary}

\begin{proof}
Since $L$ is a finite Galois extension of $K$, for any prime ideal $\mathfrak{P}$ of $\cO_L$ 
we have $f(\mathfrak{P}) \mid d$, where $d= [L:K]$. 
So, automatically we have $N_{K}(\mathfrak{p})^{f(\mathfrak{P})}-1\mid
N_{K}(\mathfrak{p})^{d}-1$ for any prime ideal $\mathfrak{p}$ of $\cO_K$ lying below $\mathfrak{P}$.  
The rest follows from Theorem~\ref{thm:Ded-ideal} and definition. 
\end{proof}

\section{Carmichael polynomials over finite fields}
\label{sec:Car-pol}

In this section, we study Carmichael polynomials in $\F_q[t]$. 

A Korselt-type criterion for Carmichael polynomials follows directly from Theorem~\ref{thm:Ded-Korselt}. 

\begin{theorem}[Korselt's criterion for polynomials]  \label{thm:pol-Korselt} 
A composite polynomial $g \in \F_q[t]$ is a Carmichael polynomial  if and only if 
\begin{itemize}
\item[(1)] $g$ is square-free,
\item[(2)] for any irreducible factor $P$ of $g$, $\deg P \mid \deg g$. 
\end{itemize}
\end{theorem}

\begin{proof}
We only need to mention the second condition. 
When $g$ is a Carmichael polynomial, then by Theorem~\ref{thm:Ded-Korselt}, 
for any irreducible factor $P$ of $g$ we have that $q^{\deg P} -1$ divides $q^{\deg g} -1$, 
which is equivalent to $\deg P \mid \deg g$.   
\end{proof}

From Theorem~\ref{thm:pol-Korselt}, we know that any polynomial of prime degree greater than $q$ is not a Carmichael polynomial. 
It is also easy to see that there are infinitely many Carmichael polynomials in $\F_q[t]$. 
Besides, for any integer $m \ge 2$, there are infinitely many Carmichael polynomials having exactly $m$ irreducible monic factors; 
for example, one can choose polynomials having exactly $m$ irreducible monic factors of the same degree.

In fact,  we can construct Carmichael polynomials starting from any square-free polynomial. 
However, the analogue is not true for Carmichael numbers (because all Carmichael numbers are odd). 

\begin{theorem}  \label{thm:const}
Let $u \in \F_q[t]$ be a square-free polynomial. 
Let $g, h \in \F_q[t]$ satisfy $g \ne 0$ and $\gcd(g,h)=1$. 
Then, there are infinitely many square-free monic polynomials $w$ whose irreducible monic factors are all congruent to $h$ modulo $g$ 
 such that $uw$ are Carmichael polynomials. 
\end{theorem}

\begin{proof}
Let $m$ be the least common multiple of $\deg u$ and the degrees of all the irreducible factors of $u$. 
By Dirichlet's theorem on primes in   arithmetic progressions in
$\mathbb{F}_{q}[t]$ (see \cite[Theorem 4.8]{Rosen}), 
we know that for any sufficiently large integer $d$, in the arithmetic progression $h$ modulo $g$
 there exist $dm-\deg u - 1$ irreducible monic polynomials 
$P_1, \ldots, P_k ~(k=dm-\deg u -1)$ of degree $dm$ and an irreducible monic polynomial $Q$ of degree $dm-\deg u$. 
Then, we obtain square-free polynomials $uP_1 \cdots P_k Q$ of degree $dm(dm-\deg u)$, 
which are Carmichael polynomials by Theorem~\ref{thm:pol-Korselt}. 
\end{proof}

As a consequence, we can confirm the infinitude of Carmichael polynomials in arithmetic progressions. 

\begin{corollary}  \label{cor:ap}
Given two polynomials $g,h\in \F_q[t]$ with $g \ne 0$, 
assume that $\gcd(g,h)$ is either equal to $1$ or square-free. 
Then, there are infinitely many monic Carmichael polynomials congruent to $h$ modulo $g$. 
\end{corollary}

\begin{proof}
By assumption and using Dirichlet's theorem on primes in   arithmetic progressions in
$\mathbb{F}_{q}[t]$, 
we have that for any sufficiently large integer $d$, there are square-free monic polynomials $u \in \F_q[t]$ 
of degree $d$ such that $u \equiv h ~({\rm mod}~g)$. 
Fix such a polynomial $u$. 
By Theorem~\ref{thm:const}, we see that 
there are infinitely many square-free monic polynomials $w$ whose irreducible monic factors are all congruent to $1$ modulo $g$ 
 such that $uw$ are Carmichael polynomials.
By construction, we have $uw \equiv h ~({\rm mod}~g)$.   
This completes the proof. 
\end{proof}

We remark that in Corollary~\ref{cor:ap}, if $\gcd(g,h)=1$, 
then for any sufficiently large integer $d$, we can construct such Carmichael polynomials of the form $P_1P_2$, 
where $P_1, P_2$ are irreducible monic polynomials of the same degree satisfying $P_1 \equiv h ~({\rm mod}~g)$ and $P_2 \equiv 1 ~({\rm mod}~g)$. 

However, it is not true that for any composite integer $n$, there exist Carmichael polynomials of degree $n$.  
We can confirm this explicitly and further obtain some quantitative results. 

We first make some preparations. 

For any integer $n \ge 1$,  let $\pi_q(n)$ be the number of monic  irreducible polynomials of degree $n$
 in $\F_q[t]$. 
 It is well-known that (for instance, see \cite[Corollary of Proposition 2.1]{Rosen})
\begin{equation}  \label{eq:piq1}
 \pi_q(n) = \frac{1}{n} \sum_{d \mid n}  \mu(d) q^{n/d}, 
\end{equation}
where $\mu$ is the M{\"o}bius function. 
By \cite[Lemma 4]{Pollack}, we have
\begin{equation}
\label{eq:piq2}
 \frac{q^n}{n} - 2\frac{q^{n/2}}{n} \le \pi_q(n) \le \frac{q^n}{n},  \qquad \pi_q(n) \ge \frac{q^n}{2n}. 
\end{equation}

Moreover, we have:

\begin{lemma}  \label{lem:piq}  
If $q\ge 4$, $\pi_q(n)$ is strictly increasing with respect to $n \ge 1$. 
Besides, both $\pi_2(n)$ and $\pi_3(n)$ are  strictly increasing with respect to $n \ge 2$. 
\end{lemma}

\begin{proof}
If $q \ge 5$, then  for any $n \ge 1$, using \eqref{eq:piq2} we have 
$$
\pi_q(n) \le \frac{q^n}{n}  < \frac{q^{n+1}}{2(n+1)} \le \pi_q(n+1). 
$$

If $q=4$, we similarly have for any $n \ge 2$, 
$$
\pi_4(n) \le \frac{4^n}{n}  < \frac{4^{n+1}}{2(n+1)} \le \pi_4(n+1). 
$$
From \eqref{eq:piq1} we directly have $\pi_4(1)=4$ and $\pi_4(2)=6$, and so $\pi_4(1) < \pi_4(2)$.  

If $q=3$, we again have for any $n \ge 3$, 
$$
\pi_3(n) \le \frac{3^n}{n}  < \frac{3^{n+1}}{2(n+1)} \le \pi_3(n+1). 
$$
Using \eqref{eq:piq1}, we get $\pi_3(1)=3, \pi_3(2)=3$ and $\pi_3(3)=8$, 
and thus $\pi_3(2) < \pi_3(3)$.  

If $q=2$, using \eqref{eq:piq2} we also have for any $n \ge 4$, 
$$
\pi_2(n) \le \frac{2^n}{n}  < \frac{2^{n+1}}{n+1} - 2\frac{2^{(n+1)/2}}{n+1} \le \pi_2(n+1). 
$$
From \eqref{eq:piq1} we obtain $\pi_2(1)=2, \pi_2(2)=1, \pi_2(3)=2$ and $\pi_2(4)=3$, 
and so $\pi_2(2) < \pi_2(3) < \pi_2(4)$.  
\end{proof}

For any integer $n \ge 1$, let $C_q(n)$ be the number of monic Carmichael polynomials in $\F_q[t]$ of degree $n$. 
By Theorem~\ref{thm:pol-Korselt}, if $n$ is a prime number and $n \le q$, 
then considering the product of $n$ distinct linear monic polynomials, we have 
$$
C_q(n) = \binom{q}{n}; 
$$
otherwise if $n$ is a prime and $n>q$, we have $C_q(n)=0$. 

\begin{theorem}  \label{thm:lower}
Let $n$ be a composite integer and $\ell$ the smallest prime factor of $n$. 
Then, $C_q(n) = 0$ if and only if $(q,n)=(2,9)$. 
If $(q,n) \ne (2,9)$, then 
$C_q(n) = 1$ if and only if $(q,n)=(2,4)$;  
and moreover, we have 
$$
C_q(n)  \ge  \frac{q^n}{(2n)^\ell}.
$$
\end{theorem}

\begin{proof}

Since $\pi_2(1)=2, \pi_2(2)=1$ and $\pi_2(3)=2$, by Theorem~\ref{thm:pol-Korselt} we have $C_2(4)=1, C_2(9) = 0$. 

If $\pi_q(n/\ell) > \ell$, then 
we can choose polynomials $g$ to be the product of $\ell$ distinct irreducible monic polynomials of degree $n/\ell$. 
By Theorem~\ref{thm:pol-Korselt}, they are Carmichael polynomials. 
Counting these polynomials, we have 
\begin{equation}  \label{eq:Cqn}
C_q(n)  \ge \binom{\pi_q(n/\ell)}{\ell} \ge \pi_q(n/\ell) > \ell \ge 2. 
\end{equation}  
So, it remains to find the condition when $\pi_q(n/\ell) > \ell$. 

If $q\ge 3$, using \eqref{eq:piq2} and noticing $q^m > 2m^2$ for any integer $m \ge 1$,  we obtain  
$$
\pi_q(n/\ell)  \ge \frac{q^{n/\ell}}{2n/\ell} > \frac{2(n/\ell)^2}{2n/\ell} = n/\ell \ge \ell. 
$$

Similarly, if $q=2$, using \eqref{eq:piq2} and noticing $2^m > 2m^2$ for any integer $m \ge 7$,  
we obtain  for $n / \ell \ge 7$,  
$$
\pi_2(n/\ell)  \ge \frac{2^{n/\ell}}{2n/\ell} > \frac{2(n/\ell)^2}{2n/\ell} = n/\ell \ge \ell. 
$$
If $n / \ell \le 6 $, then $\ell \le 6$, and so $n \le 36$. 
Thus, we only need to consider composite integers $n \le 36$. 
There are only three cases $\ell =2, 3, $ or $5$. 

If $\ell = 2$ and $n \ge 8$, by Lemma~\ref{lem:piq} we have $\pi_2(n/2) \ge \pi_2(4)=3 > 2$.  

If $\ell = 3$ and $n \ge 15$,  
 by Lemma~\ref{lem:piq} we have $\pi_2(n/3) \ge \pi_2(5)=6 > 3$.  
 
Now, if $\ell = 5$, then $n \ge 25$, 
 and we have $\pi_2(n/5) \ge \pi_2(5)=6 > 5$.  
 
So, it remains to consider $n=6$ when $q=2$. 
By \eqref{eq:piq1}, it is easy to see that  $\pi_2(6/2)=2$ and $C_2(6) = 5$.  

Hence, $\pi_q(n/\ell) > \ell$ (and so \eqref{eq:Cqn}) holds for $q\ge 3$, or $q=2$ and composite $n \ne 4,6,9$. 

Collecting the above considerations, if $(q,n) \ne (2,4), (2,9)$, 
 then $\pi_q(n/\ell) \ge \ell$, 
and so, by \eqref{eq:Cqn} we have 
$$
C_q(n)  \ge \binom{\pi_q(n/\ell)}{\ell}, 
$$
which, together with \eqref{eq:piq2}, implies that 
$$
C_q(n) \ge \big( \pi_q(n/\ell)/\ell \big)^\ell  \ge  q^n / (2n)^\ell. 
$$
This inequality also covers the case $(q,n) = (2,4)$ since $C_2(4)=1$.  
\end{proof}

Now, we want to get an upper bound for $C_q(n)$, 
which implies that the natural density of Carmichael polynomials is zero. 

\begin{theorem}  \label{thm:upper}
Let $n$ be a composite number. 
Then, for any $0< \varepsilon < 1/2$, 
there exists a contant $c(q,\varepsilon)$ such that 
if $n > c$, we have 
$$
C_q(n)  \le  \frac{q^n}{n^{1/2-\varepsilon}}.  
$$
\end{theorem}

\begin{proof}
We first arrange all the proper factors $d_1, \ldots, d_r$ of $n$ as follows:
$$
1=d_1 < d_2 < \cdots < d_r < n,
$$
where $r$ is the number of proper factors of $n$. 
We define a subset of $r$-tuples of non-negative integers: 
$$
T(n) = \{(k_1,\ldots,k_r): \, k_1d_1 + \cdots + k_rd_r = n, k_1 \le q \}. 
$$
Note that since $d_1=1$, for each tuple 
$(k_1,\ldots,k_r)$ in $T(n)$, $k_1$ is fixed when $k_2, \ldots, k_r$ are all fixed. 

For any monic Carmichael polynomial of degree $n$, by definition 
the degree of each of its irreducible monic factors divides $n$, 
and so it corresponds to one tuple in $T(n)$ by collecting the degrees of its irreducible factors.  
Conversely, every tuple $(k_1,\ldots,k_r)$ in $T(n)$ corresponds to 
$$
\binom{\pi_q(d_1)}{k_1} \cdots \binom{\pi_q(d_r)}{k_r}
$$
distinct monic Carmichael polynomials of degree $n$. 

Hence, using \eqref{eq:piq2} we obtain    
\begin{equation}  \label{eq:Cqn2}
\begin{split}
C_q(n) & = \sum_{(k_1,\ldots,k_r) \in T(n)} \binom{\pi_q(d_1)}{k_1} \cdots \binom{\pi_q(d_r)}{k_r}  \\
& \le \sum_{(k_1,\ldots,k_r) \in T(n)} \pi_q(d_1)^{k_1} \cdots \pi_q(d_r)^{k_r} \\
& \le \sum_{(k_1,\ldots,k_r) \in T(n)} \frac{q^{k_1d_1}}{d_1^{k_1}} \cdots \frac{q^{k_rd_r}}{d_r^{k_r}} 
= q^n \sum_{(k_1,\ldots,k_r) \in T(n)} \frac{1}{d_2^{k_2} \cdots d_r^{k_r}}.  
\end{split}
\end{equation}
So, it remains to estimate the summation 
$$
S(n) = \sum_{(k_1,\ldots,k_r) \in T(n)} \frac{1}{d_2^{k_2} \cdots d_r^{k_r}}.
$$ 

Note that for each tuple $(k_1,\ldots,k_r) \in T(n)$, we have $k_i \le n/d_i$ for each $2 \le i \le r$ and 
\begin{equation}  \label{eq:k2kr}
n-q \le k_2d_2 + \cdots + k_rd_r \le n. 
\end{equation} 
Put 
\begin{equation*}
W(n) = \prod_{i=2}^{r} (1+ \frac{1}{d_i} + \frac{1}{d_i^2} + \cdots + \frac{1}{d_i^{n/d_i}}). 
\end{equation*}
Clearly, $S(n)$ is a part of the summation $W(n)$ (after expanding the products).  
In the sequel, we estimate $S(n)$ by distinguishing the main part of $W(n)$. 

To estimate $W(n)$, we first have 
\begin{equation*}
\begin{split}
 \log W(n) & < \log \prod_{i=2}^{r} \frac{1}{1-1/d_i}  = -\sum_{i=2}^{r} \log (1-1/d_i) \\
 & = \sum_{i=2}^{r} \Big(\frac{1}{d_i} + \frac{1}{2d_i^2} + \frac{1}{3d_i^3} + \cdots \Big) \\
 & < \sum_{i=2}^{r} \Big(\frac{1}{d_i} + \frac{1}{d_i^2}\Big) \\
 &  \le \frac{\sigma(n)}{n} -1 - \frac{1}{n} + \int_{1}^{n} x^{-2} \, \mathrm{d} x  < \sigma(n) / n, 
\end{split}
\end{equation*}
where $\sigma(n)$ as usual is the sum of all the factors of $n$.  
Using a classical result of Robin \cite[Th{\'e}or{\`e}me 2]{Robin} that 
\begin{equation*}
\frac{\sigma(n)}{n} < \exp(\gamma) \log\log n + \frac{0.6483}{\log\log n}, \quad n \ge 3, 
\end{equation*}
where $\gamma$ is the Euler-Mascheroni constant ($\gamma=0.577215664901532\ldots$), 
we directly have for $n \ge 268$
\begin{equation*}
\frac{\sigma(n)}{n} < 2 \log\log n - 0.2189 \log\log n + \frac{0.6483}{\log\log n} < 2\log\log n. 
\end{equation*}
Hence,  we obtain 
\begin{equation}  \label{eq:W(n)}
W(n) < (\log n)^2, \qquad n \ge 268. 
\end{equation}

Now, we want to find the main part of $W(n)$. 
For a fixed $0 < \varepsilon < 1/2$,  let $j \ge 1$ be the unique index satisfying  
\begin{equation*}
  1=d_1 < d_2 < \cdots < d_j <  (n-q)^{(1-\varepsilon)/2} \le d_{j+1} < \cdots < d_r. 
\end{equation*}
For each $2 \le i \le r$, let 
$$
m_i = \lfloor ((n-q)/d_i)^{2\varepsilon /(1+ \varepsilon)} \rfloor. 
$$
Then, since 
\begin{equation*}
\begin{split}
\sum_{i=2}^{j} m_i d_i 
& \le (n-q)^{2\varepsilon / (1+ \varepsilon)} \sum_{i=2}^{j} d_i^{(1- \varepsilon)/(1+ \varepsilon)} \\ 
& < (n-q)^{2\varepsilon / (1+ \varepsilon)} \int_{1}^{(n-q)^{(1- \varepsilon)/2}} x^{(1- \varepsilon)/(1+ \varepsilon)} \, \mathrm{d} x \\
& < (n-q)^{2\varepsilon / (1+ \varepsilon)} \cdot (n-q)^{(1- \varepsilon) / (1+ \varepsilon)} = n-q, 
\end{split}
\end{equation*}
in view of \eqref{eq:k2kr} we know that any summation term of 
$$
V(n) = \prod_{i=2}^{j} (1+ \frac{1}{d_i} + \frac{1}{d_i^2} + \cdots + \frac{1}{d_i^{m_i}}) 
$$
(after expanding the products) does not appear in $S(n)$. 
Thus, we have 
\begin{equation}  \label{eq:SWV}
S(n) \le W(n) - V(n). 
\end{equation}
It suffices to estimate $W(n)-V(n)$.  

For each $2 \le i \le j$, we have 
\begin{equation}   \label{eq:i<j}
\begin{split}
\frac{1}{d_i^{m_i + 1}} + \frac{1}{d_i^{m_i +2}} + \cdots + \frac{1}{d_i^{n/d_i}} 
& < \frac{1/ d_i^{m_i + 1}}{1-1/d_i}  = \frac{1}{(d_i -1)d_i^{m_i}} \\
& \le 2^{-m_i} \le  2^{1-(n-q)^{\varepsilon}}. 
\end{split}
\end{equation}
On the other hand, for each $j+1 \le i \le r$ we have 
\begin{equation}  \label{eq:i>j}
\begin{split}
\frac{1}{d_i} + \frac{1}{d_i^{2}} + \cdots + \frac{1}{d_i^{n/d_i}} 
& < \frac{1/ d_i}{1-1/d_i}  = \frac{1}{d_i -1} \\
& \le  \frac{1}{(n-q)^{(1- \varepsilon)/2} - 1}. 
\end{split}
\end{equation}  
Therefore, combining \eqref{eq:i<j}, \eqref{eq:i>j} with \eqref{eq:W(n)}, we deduce that 
\begin{equation}  \label{eq:WV1}
\begin{split}
W(n) - V(n) & \le 
\sum_{i=2}^{j}  \Big( \frac{1}{d_i^{m_i + 1}} + \frac{1}{d_i^{m_i +2}} + \cdots + \frac{1}{d_i^{n/d_i}} \Big) W(n)  \\
& \quad + \sum_{i=j+1}^{r}  \Big( \frac{1}{d_i} + \frac{1}{d_i^2} + \cdots + \frac{1}{d_i^{n / d_i}} \Big)W(n)  \\
& \le \Big( 2^{1-(n-q)^{\varepsilon}}  + ((n-q)^{(1- \varepsilon)/2}-1)^{-1} \Big) (\log n)^2 \tau(n), 
\end{split}
\end{equation}
where $\tau(n)$ is the number of factors of $n$. 
For $\tau(n)$, a classical result of Wigert says that (see, for instance, \cite[Theorem 13.12]{Apostol})
\begin{equation*}
\tau(n) = n^{O(1/\log\log n)}. 
\end{equation*}
Hence, for sufficiently large $n$ (depending on $q,\varepsilon$), \eqref{eq:WV1} becomes 
\begin{equation}  \label{eq:WV}
W(n) - V(n) \le  n^{-1/2 + \varepsilon}. 
\end{equation} 
Finally, the desired result follows from \eqref{eq:Cqn2},  \eqref{eq:SWV} and \eqref{eq:WV}.   
\end{proof}

\begin{corollary}
The natural density of Carmichael polynomials in $\F_q[t]$ is zero. 
That is, we have 
$$
\lim_{n \to \infty} \frac{C_q(1)+C_q(2)+ \cdots + C_q(n)}{q^n} = 0. 
$$
\end{corollary}

Finally, we extend the concept of Carmichael polynomials as the integer case. 

Recall that for any integer $d \ge 1$, a \textit{rigid Carmicahel number of order $d$} is a composite square-free integer $n$ satisfying 
$p^{i}-1\mid n^{d}-1$  for all primes $p\mid n$ and all $1 \le  i \leq d$ 
(see \cite{Howe} or the  comments after Theorem 2.7 in~\cite{Steele}). 
It is conjectured that there are infinitely many rigid Carmichael numbers of order $d$ for any $d \ge 2$. 

Similarly, we define a \textit{rigid Carmichael polynomial of order $d$} in $\F_q[t]$ to be 
a reducibe square-free polynomial $g \in \F_q[t]$ satisfying 
$ i \deg P \mid d \deg g$ for any irreducible polynomial 
$P$ dividing $g$ and any $i = 1,\ldots d$.  
For example, let $g=P_{1}P_{2}$ with $\textrm{deg}P_{i} = 3$
and $d = 3$, then $g$ is a Carmichael polynomial of order 3 in
$\mathbb{F}_{q}[t]$.

\begin{theorem}  \label{thm:rigid} 
For any positive integer $d$, there exist infinitely many rigid monic Carmichael polynomials  of order $d$ in 
$\F_{q}[t]$. 
\end{theorem}

\begin{proof}  
We only need to consider the case when $d \ge 2$.
Fix a positive integer $d \ge 2$. Let $m$ be the least common multiple of $1,2,\ldots, d$.
For any positive integer $n$ satisfying $\pi_q(n) \ge m$, we can construct polynomials $g=P_1 \cdots P_m$, 
where  $P_1, \ldots, P_m$ are distinct monic irreducible polynomials of degree $n$. 
Then, $\deg g= mn$. Thus, for  any  $1\leq j \leq m$ and $1 \le i \leq d$, we have $i \deg P_{j}=in \mid
d \deg g=dmn$, and so $g$ is a rigid Carmichael polynomial of order $d$. 
Letting  $n$ go to $\infty$, we get infinity many such polynomials $g$. 
This completes the proof. 
\end{proof}

\section{Carmichael elements in function fields }  
\label{sec:Car-ele}

Let $K$ be a function field (that is, a finite extension over $\F_q(t)$), and let $\cO_K$ be the ring of integers of $K$. 
We say that an element $\alpha \in \cO_K$ is Carmichael in $K$ 
if $\alpha$ is a Carmichael element of $\cO_K$ (see Definition~\ref{def:Car-ideal}).  

In this section, as the number field case \cite{Steele}, we consider the following questions: 
\begin{itemize}
\item[(1)]  For any function field $K$, does it have infinitely many Carmichael elements? 

\item[(2)] For any square-free polynomial $g$ in $\F_q[t]$, is it Carmichael in infinitely many function fields with discriminant relatively prime to $g$? 

\item[(3)] For any square-free polynomial $g$ in $\F_q[t]$, is it not Carmichael in infinitely many function fields with discriminant relatively prime to $g$? 
\end{itemize}
We give a definite answer to the first question (see Corollary~\ref{cor:Car-K} below) 
and some partial answers to the second and third questions whose answers we conjecture are both positive.  
 
First, we consider the case of Carmichael polynomials in $\F_q[t]$. 

\begin{theorem}  \label{thm:ext-Car1}
Let $g$ be a Carmichael polynomial in $\F_q[t]$. 
Then, $g$ is Carmichael in any finite Galois extension over $\mathbb{F}_{q}(t)$ with 
discriminant relatively prime to $g$.
\end{theorem}

\begin{proof}
Suppose that $K$ is  a finite Galois extension over $\mathbb{F}_{q}(t)$
with degree $d$ and discriminant ${\rm Disc}(K)$. 
For any irreducible factor $P$ of $g$, 
let $f(P)$ be the residue class degree of $P$ in $K/\mathbb{F}_{q}(t)$. 
Due to the choice of $K$, we have $f(P) \mid d$.   
Since $g$ is relatively prime to Disc$(K)$, each irreducible factor of $g$ is unramified in $K/\F_q(t)$.
Note that $g$ is a Carmichael polynomial in $\F_q[t]$. 
Then, the ideal $g\mathcal{O}_{K}$ is square-free, 
and for any irreducible factor $P$ of $g$, we have $\deg P\mid \deg g$.  
Given a prime ideal  $\mathfrak{p}$ of $K$ lying above $P$, we have 
$$
N_K(\mathfrak{p}) = N_{\F_q(t)}(P)^{f(P)} = q^{f(P)\deg P}. 
$$
Then, noticing $N_K(g\cO_K) = q^{d \deg g}$ and $f(P)\deg P \mid d \deg g$, we have 
$$
N_K(\mathfrak{p}) - 1 \mid  N_K(g\cO_K) - 1. 
$$
Hence, from Theorem~\ref{thm:Ded-Korselt}, $g$ is Carmichael in $K$.
\end{proof}

We remark that the number field case does not have a similar result as the above theorem; 
see \cite[Theorem 3.1]{Steele}.  

However, a Carmichael polynomial might not be Carmichael in infinitely many function fields. 
More generally, we have: 

\begin{theorem}  \label{thm:ext-Car2}
Let $g$ be a square-free polynomial in $\F_q[t]$ of odd degree. 
Assume that $3 \nmid q$ and  $3 \nmid q-1$.  
Then, $g$ is not Carmichael in infinitely many cubic function fields over $\F_q(t)$ with 
discriminant  relatively prime to $g$.
\end{theorem}

\begin{proof}
By assumption, we can choose an irreducible factor, say $P$, of $g$ such that the degree $\deg P$ is odd. 
Noticing that $3 \nmid q$ and  $3 \nmid q-1$, we have $3 \mid q^{\deg P} +1$. 
We choose two distinct irreducible polynomials $G,H \in \F_q[t]$ not dividing $g$ such that 
$$
G \equiv H ~({\rm mod}~P).  
$$
Let $D=GH^2$. So, $D$ is a cube modulo $P$. 
Let $K$ be the cubic function field generated by $\sqrt[3]{D}$, which is  a cubic root of $D$ over $\F_q(t)$. 
Then, the discriminant of $K/\F_q(t)$ is $-27G^2H^2$ (see \cite[page 610]{Sch}), 
which is indeed relatively prime to $g$.  
Then, by \cite[Theorem 3.1]{Sch}, we have that $P\cO_K$ is a product of two distinct prime ideals in $\cO_K$, 
say $\mathfrak{p}_1$ and $\mathfrak{p}_2$. 
So, for the residue class degrees $f(\mathfrak{p}_1)$ and $f(\mathfrak{p}_2)$, one of them is equal to 2, say $f(\mathfrak{p}_1)$. 
Clearly, $N_K(\mathfrak{p}_1)=q^{f(\mathfrak{p}_1)\deg P}=q^{2\deg P}$, 
and $N_K(g\cO_K) = q^{3\deg g}$. 
Noticing $2 \nmid 3\deg g$, we have 
$$
N_K(\mathfrak{p}_1) - 1 \nmid  N_K(g\cO_K) - 1, 
$$
which implies that $g$ is not Carmichael in $K$ by Theorem~\ref{thm:Ded-Korselt}. 
We conclude the proof by noticing that there are infinitely many choices of polynomials $G,H$. 
\end{proof}

Similar as Theorem~\ref{thm:ext-Car1}, we have: 

\begin{theorem}  \label{thm:ext-rigid}
Let $g$ be a rigid Carmichael polynomial of order $d$ in
$\mathbb{F}_{q}[t]$. Then, $g$ is Carmichael in any finite extension
over $\mathbb{F}_{q}(t)$ with degree $d$ whose discriminant is relatively
prime to $g$.
\end{theorem}

We now answer the question about the infinitude of Carmichael elements in any function field. 

\begin{corollary}  \label{cor:Car-K}
For any finite extension $K$ over $\F_q(t)$, there are infinitely many Carmichael elements in $K$. 
\end{corollary}

\begin{proof}
Fix a positive integer $d \ge 2$. 
Let $K$ be an arbitrary  finite extension over $\mathbb{F}_{q}(t)$ of degree $d$. 
Let $m$ be the least common multiple of $1, \ldots, d$. 
Denote by $S(m)$ the set of polynomials which are
 the product of $m$ distinct irreducible polynomials of the same degree. 
As in the proof of Theorem~\ref{thm:rigid}, each polynomial in $S(m)$ 
is a rigid Carmichael polynomial of order $d$ in $\mathbb{F}_{q}[t]$. 
Obviously, there are infinitely many polynomials in $S(m)$ relatively prime to ${\rm Disc}(K)$.  
We conclude the proof by using Theorem~\ref{thm:ext-rigid}.  
\end{proof}

From now on, we consider the case of non-Carmichael square-free polynomials in $\F_q[t]$. 

The following result suggests that a non-Carmichael square-free polynomial 
 can be Carmichael in infinitely many functions fields.

\begin{theorem}  \label{thm:Kummer} 
Let $g \in\mathbb{F}_{q}[t]$ be a square-free polynomial. 
Let $\ell$ be any prime factor of $q-1$ (it requires $q \ge 3$). 
Let $P_{i} ~ (1\leq i\leq s)$ be all the
monic irreducible factors of $g$ whose degrees do not divide the degree of $g$,
and we further assume that $\deg P_{i}= \ell~(1\leq i\leq s)$. 
 Then, there exist infinitely many 
  cyclic extensions of degree $\ell$ whose discriminants are relatively
prime to $g$ such that $g$ is Carmichael in them.
\end{theorem}

\begin{proof}
From Dirichlet's theorem on primes in   arithmetic progressions in
$\mathbb{F}_{q}[t]$ (see \cite[Theorem 4.8]{Rosen}), there exist
infinitely many irreducible monic polynomials $Q$ of even degree such that $Q$ is
relatively prime to $g$ and
$$
\left(\frac{P_{i}}{Q}\right)_{\ell}=1 ~(1\leq i\leq s),
$$
where $\left(\frac{\cdot}{\cdot}\right)_{\ell}$ be the $\ell$-th power residue
symbol in $\mathbb{F}_{q}[t]$ (see \cite[page 24]{Rosen}). 
From the $\ell$-th power reciprocity law in $\mathbb{F}_{q}[t]$ (see \cite[Theorem 3.3]{Rosen}), 
we have $\left(\frac{Q}{P_{i}}\right)_{\ell}=1~(1\leq
i\leq s)$ by noticing $\deg Q$ is even. 
Using \cite[Proposition 10.5]{Rosen},  each $P_{i} ~(1\leq i\leq s)$ splits completely in
$K=\mathbb{F}_{q}(t)(\sqrt[\ell]{Q})$. 
Thus, if $\mathfrak{p}$ is any prime factor of
$g\cO_{K}$ lying above some $P_{i}~ (1\leq i\leq s)$, we have
$f(P_{i})=1$ and 
$$
N_{K}(\mathfrak{p})-1=q^{\textrm{deg}
P_{i}}-1=q^{\ell}-1\mid q^{\ell\deg g}-1 = N_{K}(g\cO_K)-1.
$$ 
If $\mathfrak{p}$ is any prime factor
of $g\cO_{K}$ lying above a monic irreducible factor $P$ of $g$ such that 
  $P \ne P_{i}~~(1\leq i\leq s)$. 
Then, we have $\deg P \mid \deg g$ by assumption,  and so 
$$
N_{K}(\mathfrak{p})-1=q^{f(P)\deg P}-1 \mid
q^{\ell \deg g}-1= N_{K}(g\mathcal{O}_{K})-1, 
$$ 
where $f(P)$ is  the residue class degree of $P$ in $K/\mathbb{F}_{q}(t)$ and $f(P) \mid \ell$.    
Hence, by Theorem~\ref{thm:Ded-Korselt}, $g$ is Carmichael in $K$.
\end{proof}

As one can imagine, a non-Carmichael square-free polynomial in $\F_q[t]$ 
is more likely not to be Carmichael in infinitely many function fields. 
We confirm this by constructing two kinds of function fields: Kummer function fields and cyclotomic function fields. 

\begin{theorem}  \label{thm:cyclic}
Let $g \in\mathbb{F}_{q}[t]$ be a non-Carmichael square-free polynomial. 
Let $\ell$ be any prime factor of $q-1$. 
Then, there exists infinitely
many cyclic extensions of degree $\ell$ whose discriminants are relatively
prime to $g$ such that $g$ is not Carmichael in them.
\end{theorem}

\begin{proof}
Since $g \in\mathbb{F}_{q}[t]$ is a non-Carmichael square-free polynomial, 
by Theorem~\ref{thm:pol-Korselt} $g$ has a  monic irreducible factor, say $P$, 
such that $\deg P \nmid \deg g$. 

Let $\eta$ be a primitive  $\ell$-th root of unity in $\mathbb{F}_{q}^*$. 
As before, there exist infinitely many irreducible monic
polynomials $Q$ of even degree such that $Q$ is relatively prime to $g$ and
$\left(\frac{P}{Q}\right)_{\ell}=\eta$. 
 From the $\ell$-th power
reciprocity law of $\mathbb{F}_{q}[t]$ and noticing $\deg Q$ is even, we have $\left(\frac{Q}{P}\right)_{\ell}=\eta$. 
Using \cite[Proposition 10.5]{Rosen} and noticing $\eta \ne 1$, 
we know that $P$ is inert in $K=\mathbb{F}_{q}(t)(\sqrt[\ell]{Q})$.
For the  prime ideal  $\mathfrak{p}$ in $K$ lying above $P$, 
noticing $\deg P \nmid \deg g$ we have
$$
N_{K}(\mathfrak{p})-1=q^{\ell \deg P}-1 \nmid
q^{\ell \deg g}-1 = N_{K}(g\mathcal{O}_{K})-1.
$$
Hence, from  Theorem~\ref{thm:Ded-Korselt}, $g$ is not Carmichael in $K$.
\end{proof}

Note that Theorem~\ref{thm:cyclic} does not cover the case when $q=2$.  
We supplement this by using cyclotomic function fields. 
First we recall briefly the definition of cyclotomic function fields.

Let $\overline{\mathbb{F}_q(t)}$ be the algebraic closure of $\F_q(t)$. 
Let $\mathrm{End}(\overline{\mathbb{F}_q(t)})$ be the ring of
$\mathbb{F}_q$-algebra endomorphism of $\overline{\mathbb{F}_q(t)}$. 
Let
$$
\rho: \, \F_q[t] \to \mathrm{End}(\overline{\mathbb{F}_q(t)}),\quad  M \mapsto \rho_M
$$ 
be the ring homomorphism defined by
$$
\rho_a(\alpha)=a\alpha,\quad  \rho_t(\alpha)=t\alpha+\alpha^q,
$$
where $a\in \mathbb{F}_{q}$ and $\alpha\in \overline{\mathbb{F}_q(t)}$. 
For any non-constant polynomial $M \in \F_q[t]$, define 
$$
\Lambda_{M}=\{\alpha\in \overline{\mathbb{F}_q(t)}: ~\rho_M(\alpha)=0\}. 
$$ 
Then, the function field generated by $\Lambda_{M}$ over $\F_q(t)$ is called the $M$-th cyclotomic function
field, denoted by $\F_q(t)(\Lambda_M)$. 
Note that the degree of $\F_q(t)(\Lambda_M)$ over $\F_q(t)$ is equal to $\Phi(M)=|(\F_q[t]/M\F_q[t])^*|$, 
where $\Phi$ is the Euler $\phi$-function in $\mathbb{F}_{q}[t]$ (see \cite[page 5]{Rosen}).
In \cite[Chapter 12]{Rosen} and \cite[Chapter 12]{Villa} there are nice expositions to the arithmetic of cyclotomic function fields.

We also need a result of  Bilharz \cite{Bilharz} on Artin's primitive root conjecture in function fields;
  see \cite[Chapter 10]{Rosen} for more details.

\begin{theorem}[Bilharz]  \label{thm:Bilharz}
Let $K$ be a function field and $\alpha$ an element of $K^{*}$. 
Then, there are infinitely many prime ideals 
$\mathfrak{p}$ in $K$ for which $\alpha$ is a primitive root provided
that there is no prime factor $\ell$ of $q-1$ such that $\alpha$ is an
$\ell$-th power. 
\end{theorem}

We are now ready to present our final result. 

\begin{theorem}  \label{thm:cyclotomic}
Let $g \in\mathbb{F}_{q}[t]$ be a non-Carmichael square-free polynomial. 
Then, there exist infinitely many cyclotomic  function fields 
whose discriminants are relatively
prime to $g$ such that $g$ is not Carmichael in them.
\end{theorem}

\begin{proof}
By assumption, $g$ has an irreducible monic factor, say $P$, such that $\deg P \nmid \deg g$. 
By Theorem~\ref{thm:Bilharz}, there exist infinitely many
irreducible monic polynomials $Q$ relatively prime to $g$ such that $P$ is
a primitive root modulo $Q$. 
Fix any such $Q$, and  let $K=\mathbb{F}_{q}(t)(\Lambda_{Q})$. 
By \cite[Theorem 12.10]{Rosen},  
  the residue class degree $f(P)$
of $P$ in $K/\mathbb{F}_{q}(t)$ is  the smallest integer such that $P^{f(P)} \equiv 1 ~({\rm mod}~Q)$. 
Note that $P$ is a primitive root modulo $Q$. 
So, $f(P) = \Phi(Q) = [K:\F_q(t)]$, and thus $P$ is inert in $K/\F_q(t)$. 
For the unique prime ideal $\mathfrak{p}$ in $K$ lying above $P$, 
noticing $\deg P \nmid \deg g$ we obtain 
$$
N_{K}(\mathfrak{p})-1=q^{\Phi(Q)\deg P}-1\nmid q^{\Phi(Q) \deg g}-1=N_{K}(g\mathcal{O}_{K})-1.
$$ 
Hence, by Theorem~\ref{thm:Ded-Korselt}, $g$ is not Carmichael in $K$.
\end{proof}

\section*{Acknowledgement}
This work was supported by the National Natural Science Foundation  of China, Grant No. 11501212. 
The research of Min Sha was also supported by the Macquarie University Research Fellowship.

\end{document}